\documentclass[a4paper,reqno]{amsart}
\usepackage[english]{babel}
\usepackage{amsmath,amsthm,amssymb,amsfonts}
\usepackage{enumitem}
\setlist[enumerate]{leftmargin=*}
\usepackage{hyperref}

\newtheorem{theorem}{Theorem}[section]  
\newtheorem{proposition}[theorem]{Proposition}   

\newtheorem{lemma}[theorem]{Lemma}
\theoremstyle{remark}

\newtheorem{remark}[theorem]{Remark}

\numberwithin{equation}{section}

\newcommand{\NN}{\mathbb{N}}
\newcommand{\ZZ}{\mathbb{Z}} 
\newcommand{\RR}{\mathbb{R}}
\newcommand{\CC}{\mathbb{C}}

\newcommand{\TT}{\mathbb{T}}

\newcommand{\opL}{\mathcal{L}}
\newcommand{\Rz}{\mathcal{R}}

\newcommand{\prd}{\mathfrak{p}}
\newcommand{\scc}{\mathfrak{s}}

\newcommand{\cK}{\mathcal{K}}

\newcommand{\defeq}{\mathrel{:=}}
\newcommand{\eqdef}{\mathrel{=:}}

\renewcommand{\approx}{\asymp}

\newcommand{\tc}{\,:\,}

\title[Heat kernel and Riesz transform for the flow Laplacian]{Heat kernel and Riesz transform for the flow Laplacian on homogeneous trees}

\author[Alessio Martini]{Alessio Martini}
\address{Dipartimento di Scienze Matematiche ``G. L. Lagrange'', Politecnico di Torino, Corso Duca degli Abruzzi 24, 10129 Torino Italy}
\email{alessio.martini@polito.it}
\author[Federico Santagati]{Federico Santagati}
\address{Dipartimento di Scienze Matematiche ``G. L. Lagrange'', Politecnico di Torino, Corso Duca degli Abruzzi 24, 10129 Torino Italy}
\email{federico.santagati@polito.it}
\author[Maria Vallarino]{Maria Vallarino}
\address{Dipartimento di Scienze Matematiche ``G. L. Lagrange'', Politecnico di Torino, Corso Duca degli Abruzzi 24, 10129 Torino Italy}
\email{maria.vallarino@polito.it}



\begin{document}
\begin{abstract}
Let $\TT_{q+1}$ denote the homogeneous tree of degree $q+1$ with the standard graph distance $d$ and the canonical flow measure $\mu$. The metric measure space $(\TT_{q+1},d,\mu)$ is of exponential growth. Let $\opL$ denote the flow Laplacian, which is a probabilistic Laplacian self-adjoint on $L^2(\mu)$. In this note, we prove some weighted $L^1$-estimates for the heat kernel associated with $\opL$ and its gradient. As a consequence, we show that the first order Riesz transform associated with the flow Laplacian on $\TT_{q+1}$ is bounded on $L^p(\mu)$, for $p \in (1,2]$ and of weak type $(1,1)$. The latter result was proved in a previous paper by Hebisch and Steger: we give a different proof that might pave the way to further generalizations.
\end{abstract}

\maketitle

\section{Introduction}
Let $\TT_{q+1}$ denote the homogeneous tree of degree $q+1$, i.e., a connected graph with no cycles where each vertex has $q+1$ neighbours, with $q\geq 2$. We identify $\TT_{q+1}$ with its set of vertices and equip it with the standard graph distance $d$, counting the number of edges along the shortest path connecting two vertices. We fix a reference point $o\in \TT_{q+1}$ and set $|x|\defeq d(x,o)$. A ray is a half infinite geodesic, with respect to the distance $d$, emanating from $o$, and the natural boundary $\Omega$ of $\TT_{q+1}$ is identified with the family of rays. We choose a \emph{mythical ancestor} $\omega_*\in \Omega$ and consider the horocyclic foliation it induces on the tree: for each vertex $x$ there exists a unique integer index $\ell(x)$, which we call the \emph{level} of $x$, indicating to which horocycle the vertex belongs. The level function is given by $\ell(x)=d(o,x\wedge \omega_*)-d(x,x\wedge \omega_*)$, where $x\wedge\omega_*$ denotes the closest point to $x$ on the ray $\omega_*$. For each vertex $x$ we define its predecessor $\prd(x)$ as the only neighbour of $x$ such that $\ell(\prd(x))=\ell(x)+1$, while $\scc(x)$ will denote the set of the remaining neighbours, the successors of $x$, whose level is $\ell(x)-1$. We introduce a partial order relation on $\TT_{q+1}$ by writing $x\geq y$ if $d(x,y)=\ell(x)-\ell(y)$. The canonical flow on $\TT_{q+1}$ is the positive function $\mu$ defined by 
\[
\mu(x) = q^{\ell(x)},\qquad x\in \TT_{q+1}. 
\]
This is called a flow because it satisfies the flow condition
\begin{equation}\label{eq: flow}
    \mu(x)=\sum_{y\in \scc(x)} \mu(y), \quad x\in \TT_{q+1}.
\end{equation}
Such flow is canonical in the sense that it equally distributes the mass of a vertex to its successors.

It is well known that the metric measure space $(\TT_{q+1},d,\mu)$ is of exponential growth, hence nondoubling and that the Cheeger isoperimetric property fails in this setting. Harmonic analysis in nondoubling settings presents major difficulties. In particular, extensions of the theory of singular integrals and of Hardy and BMO spaces have been considered on various metric measure spaces not satisfying the doubling condition, but fulfilling some measure growth assumptions or some geometric conditions, such as the isoperimetric property (see, e.g., \cite{CMM, MMV, NTV, T, To, Tol03, Ve}).

A systematic analysis on $(\TT_{q+1},d,\mu)$ was initiated in a remarkable paper by W.~Hebisch and T.~Steger \cite{hs}, where they developed an \emph{ad hoc} Calder\'on--Zygmund theory. Subsequently, in \cite{ATV2, ATV1, San} an atomic Hardy space $H^1(\mu)$ and a space $BMO(\mu)$ adapted to this setting were introduced and studied. Such theory was applied to study the boundedness properties of the spectral multipliers and the Riesz transform associated with a suitable Laplacian $\opL$, which we shall call the \emph{flow Laplacian}. To introduce it we need some more notation. 

We shall denote by $\CC^{\TT_{q+1}}$ the set of complex valued functions defined on $\TT_{q+1}$. For every function $f$ in $\CC^{\TT_{q+1}}$ we define its \emph{flow gradient} as
\begin{equation*}
 \nabla f(x)=   f(x) - f(\prd(x)), \qquad   x \in \TT_{q+1}.
\end{equation*}
The adjoint of $\nabla$ with respect to the flow measure $\mu$ is then given by 
\[
\nabla^* f(x) = f(x)-\frac{1}{q}\sum_{y\in \scc(x)}f(y), \qquad  f\in \CC^{\TT_{q+1}}, x \in \TT_{q+1}.
\]
The {\emph{flow Laplacian}} is defined by $\opL=\frac12 \nabla^* \nabla$ and is easily seen to be given by
\begin{equation}\label{flowlapdef}
    \opL f(x) = f(x)-\frac{1}{2\sqrt q}\sum_{y\sim x} \frac{\mu^{1/2}(y)}{\mu^{1/2}(x)}f(y), \qquad  f\in \CC^{\TT_{q+1}}, x\in \TT_{q+1},
\end{equation}
where $x\sim y$ means that $x$ and $y$ are neighbours. This is precisely the Laplacian studied in \cite{hs}. It is easily seen that the flow Laplacian can be expressed in terms of the combinatorial Laplacian as follows:
\begin{equation}\label{eq: relation laplacians}
    \opL= \frac{1}{1-b}\mu^{-1/2}(\Delta-bI)\mu^{1/2},
\end{equation} 
where $b= (\sqrt{q}-1)^2/(q+1)$ and
\[
  \Delta f(x)=f(x)-\frac{1}{q+1}\sum_{y\sim x} f(y), \qquad   f\in \CC^{\TT_{q+1}}, x\in \TT_{q+1}.  
\]
It is well known (see for instance \cite{CMS}) that $b$ is the bottom of the spectrum of $\Delta$ on $L^2(\#)$, where $\#$ denotes the counting measure on $\TT_{q+1}$. It immediately follows that $\opL$ has no spectral gap on $L^2(\mu)$. Indeed, the spectrum of $\opL$ is precisely $[0,2]$, see \cite[Remark 2.1]{hs}.

Equation \eqref{eq: relation laplacians} allows us to exploit some known formulas for the heat kernel of the combinatorial Laplacian to obtain estimates of the heat kernel of $\opL$. Let $H_t$ denote the integral kernel of the heat semigroup $e^{-t\opL}$, $t>0$. We can prove the following large-time weighted $L^1$-estimates for the heat kernel $H_t$ and its gradients.

\begin{theorem}\label{L1estimateHtgradHt}
For all $\varepsilon \geq 0$ there exists $C_\varepsilon>0$ such that, for all $t \geq 1$,
\begin{align}
&\sup_{y \in \TT_{q+1}} \sum_{x \in \TT_{q+1}} |H_t(x,y)| \, e^{\varepsilon d(x,y)/\sqrt{t}}\mu(x) \le {C_\varepsilon}, 
 \label{stimapesataHt} \\
&\sup_{y \in \TT_{q+1}} \sum_{x \in \TT_{q+1}} |\nabla_x H_t(x,y)| \, e^{\varepsilon d(x,y)/\sqrt{t}} \mu(x) \le \frac{C_\varepsilon}{\sqrt{t}}, 
\label{stimapesatanablaxHt} \\
&\sup_{y \in \TT_{q+1}} \sum_{x \in \TT_{q+1}} |\nabla_y H_t(x,y)| \, e^{\varepsilon d(x,y)/\sqrt{t}} \mu(x) \le \frac{C_\varepsilon}{\sqrt{t}}, 
\label{stimapesatanablayHt} \\
&\sup_{y \in \TT_{q+1}} \sum_{x \in \TT_{q+1}} |\nabla_y \nabla_x H_t(x,y)| \, e^{\varepsilon d(x,y)/\sqrt{t}} \mu(x) \le \frac{C_\varepsilon}{{t}}.
 \label{stimanablaynablaxHt}
\end{align}
Moreover, the constant in the estimates above does not depend on $q$.
\end{theorem}

The previous estimates are a key ingredient to study boundedness properties of the first order Riesz transform associated with $\opL$, which is defined by $\Rz=\nabla \opL^{-1/2}$; as usual, fractional powers of the Laplacian are defined by means of the Spectral Theorem. We shall provide a new proof of the following result. 

\begin{theorem}\label{t: teoRiesz}
The Riesz transform $\Rz$ is of weak type $(1,1)$, bounded from $H^1(\mu)$ to $L^1(\mu)$, and bounded on $L^p(\mu)$ for all $p\in (1,2]$. 
\end{theorem}
 
The weak type $(1,1)$ and the $L^p$-boundedness of $\Rz$ for $p\in (1,2]$ were first proved in \cite[Theorem 2.3]{hs}: such proof was based on spherical analysis on the homogeneous tree. The proof we give in Section \ref{s: proof} below is different and more similar to the one usually given in the study of the boundedness of Riesz transforms on Lie groups and Riemannian manifolds: it is based on a suitable weighted $L^1$-estimate of the gradient of the heat kernel of the flow Laplacian, which allows us to prove that the integral kernel of $\Rz$ satisfies an appropriate integral H\"ormander's condition adapted to this setting. As shown in \cite{ATV1,LSTV}, the same integral H\"ormander's condition also implies the $H^1(\mu) \to L^1(\mu)$ boundedness.

We point out that, in recent joint work with M.~Levi and A.~Tabacco \cite{LMSTV}, we proved that $\Rz$ is also $L^p(\mu)$-bounded for $p \in (2,\infty)$, but is unbounded from $L^{\infty}(\mu)$ to $BMO(\mu)$, and therefore the integral kernel of the adjoint operator $\Rz^*$ does not satisfy the aforementioned integral H\"ormander's condition. An  endpoint result for $p=\infty$ for the Riesz transform $\Rz$ is still an open problem. 

\medskip

The definition of Riesz transform depends on a notion of gradient on graphs, which is not unambiguous in the literature. Many authors, including Hebisch and Steger in \cite{hs}, define the ``modulus of the gradient'' of a function $f \in \CC^{\TT_{q+1}}$ as the vertex function
\begin{equation*}
  D f(x)=\sum_{y \sim x}|f(x)-f(y)|, \qquad  x \in \TT_{q+1},
\end{equation*}
and consequently the ``modulus of the Riesz transform'' as the sublinear operator $D \opL^{-1/2}$. However, the relevant boundedness properties of the Riesz transform $\Rz = \nabla \opL^{-1/2}$ considered here are equivalent to those of the operator $D\opL^{-1/2}$ studied in \cite{hs} (see \cite[Proposition 2.2]{LMSTV}).

\medskip

The boundedness of Riesz transforms on graphs has been the object of many investigations in recent years. In \cite{BaRu, fe, Ru2, Ru} the authors obtained various boundedness results for Riesz transforms on graphs satisfying the doubling condition and some additional conditions, expressed either in terms of properties of the measure or estimates for the heat kernel. In \cite{CeMe} D.~Celotto and S.~Meda showed that the Riesz transform associated with the combinatorial Laplacian is bounded from a suitable Hardy type space to $L^1$ on graphs with the Cheeger isoperimetric property. In the recent paper \cite{CCH} the authors obtained the $L^p$-boundedness of the Riesz transform for the so-called bounded Laplacians on any weighted graph and any $p\in (1,\infty)$; however, the latter results are proved only under the assumption of positive spectral gap. We remark once again that $(\TT_{q+1},d,\mu)$ is nondoubling and does not satisfy the Cheeger isoperimetric property. Moreover, the flow Laplacian $\opL$ does not have spectral gap. Hence, none of the above-mentioned results may be applied in our case and we will prove the boundedness of the Riesz transform $\Rz$ by exploiting the Calder\'on--Zygmund theory developed in \cite{hs} and using estimate of the gradient of the heat kernel of $\opL$. 

\medskip

In what follows, we write $\NN$ for the set of natural numbers, including zero, and $\NN_+$ for the set of positive integers $\NN \setminus \{0\}$.
We use the standard notation $f_1(x)\lesssim f_2(x)$ to indicate
that there exists a positive constant $C$, independent from the variable $x$ but possibly depending on
other parameters, such that $f_1(x) \le Cf_2(x)$ for every $x$. When both $f_1(x)\lesssim f_2(x)$ and $f_2(x) \lesssim f_1(x)$ are valid, we will write $f_1(x) \approx f_2(x)$.

\section{Heat kernel estimates}
Let $e^{-t\Delta}$ and $e^{-t\opL}$ be the heat semigroups of the combinatorial Laplacian $\Delta$ and of the flow Laplacian $\opL$ on $\TT_{q+1}$, respectively. We shall denote by $h_t$ and $H_t$ the heat kernels on the respective measure spaces on which the generators are self-adjoint and bounded, i.e.,
\begin{equation*}
 e^{-t\Delta} f(x)=\sum_{y \in \TT_{q+1}} h_t(x,y)f(y), \quad e^{-t\opL} f(x)=\sum_{y \in \TT_{q+1}}H_t(x,y)f(y)\mu(y),
\end{equation*}
for every $x \in \TT_{q+1}$.
By the Spectral Theorem and \eqref{eq: relation laplacians}, we obtain the following relation between the combinatorial and the flow heat semigroups,
\begin{equation}\label{eq: htHt}
   e^{-t\opL}  = \mu^{-1/2}e^{b t/(1-b)}  e^{-t/{(1-b)}\Delta}\mu^{1/2}.
\end{equation}

Observe that, when $q=1$, $H_t=h_t \eqdef h^{\ZZ}_{t}$, which is the heat kernel of the combinatorial Laplacian of $\ZZ$. We shall always assume $q\geq 2$, but we will make an extensive instrumental use of the heat kernel on $\ZZ$. It is well known that $h^{\ZZ}_{t}$ is radial and that the function $h^{\ZZ}_{t}(j)\defeq h^{\ZZ}_{t}(j,0)$ is decreasing in $j\in \NN$. 
We will use several times the fact that
\begin{equation}\label{contractive}
    \|H_t(\cdot, y)\|_{L^1(\mu)}=1  \ \ \ \forall y \in \TT_{q+1},\ \qquad  \mathrm{and} \qquad \qquad \| h_t^{\ZZ}\|_{\ell^1(\ZZ)}=1.
\end{equation}

In the next proposition we collect some fundamental results by M.~Cowling, S.~Meda and A.~G.~Setti \cite[Theorem 2.3, Lemma 2.4, Proposition 2.5]{CMS} providing exact expressions for $h_t$ and the 2-step gradient of $h^{\ZZ}_{t}$.

\begin{proposition}\label{prop: cms} The following hold for all $t>0$, $x,y \in \TT_{q+1}$, $j\in \ZZ$:
\begin{enumerate}[label=(\roman*)]
\item\label{en:cmsi} $\displaystyle h_t(x,y)= \frac{2e^{-bt}}{(1-b)t}q^{-d(x,y)/2} \sum_{k=0}^\infty q^{-k}(d(x,y)+2k+1) h^{\ZZ}_{t(1-b)}(d(x,y)+2k+1))$;
\item\label{en:cmsii} $\displaystyle h^{\ZZ}_t(j-1)-h^{\ZZ}_t(j+1)=\frac{2j}{t}h^{\ZZ}_t(j)$.
\end{enumerate}
\end{proposition} 
By means of Proposition \ref{prop: cms} \ref{en:cmsi} and \eqref{eq: htHt}, we can express the heat kernel of $\opL$ as
\begin{equation}\label{eq: H=qJ}
    H_t(x,y)= q^{-\ell(x)/2} e^{b t/(1-b)} h_{t/{(1-b)}}(x,y) q^{-\ell(y)/2 } = Q(x,y) \, J_t(d(x,y)),
\end{equation} for every $x,y\in \TT_{q+1}$, where 
\[
    Q(x,y)=q^{-(\ell(y)+\ell(x))/2},
\]
and
\begin{equation}\label{10}
   J_t(d)=\frac{2}{t} \sum_{k=0}^\infty q^{-(2k+d)/2}(d+2k+1)h^{\ZZ}_{t}(d+2k+1), \qquad d \in \NN.
\end{equation}
As $h_t^\ZZ$ is nonnegative and decreasing on $\NN$, it is easy to see that the previous series is comparable to its first summand, that is,
\begin{equation}\label{eq: approx1}
    J_t(d)\approx   q^{-d/2} \frac{d+1}{t}h^{\ZZ}_{t}(d+1) \qquad \forall d \in \NN,
\end{equation}
and the implicit constants in \eqref{eq: approx1} do not depend on $q$.

\subsection{Weighted estimates of \texorpdfstring{$h^{\ZZ}_t$}{htZ}}
In this subsection, we shall prove weighted uniform and $\ell^1$-estimates for the heat kernel associated to the combinatorial Laplacian on $\ZZ$. To do so, we need the following technical lemma. 

\begin{lemma}\label{l: stimaphi}
Let $\varphi: \RR^+ \to \RR$ be the function defined by 
\begin{equation}\label{fi}
    \varphi(x)=-x+(1+x^2)^{1/2}+\log \bigg(\frac{x}{1+\sqrt{1+x^2}}\bigg), \qquad x>0.
\end{equation}
Then  
\begin{equation}\label{eq:unif_log_est}
\varphi(x)\le \log x +1-\log 2,\qquad \forall x>0.
\end{equation}
Moreover, for every $t_0>0$ there exists a positive constant $C_0$ depending on $x_0$ such that 
\begin{equation}\label{eq:decay_est}
\varphi(x)\le  -\frac{C_0}{x} \qquad \forall x\ge x_0.
\end{equation}
\end{lemma}
\begin{proof}
We can easily see that
\[
\varphi(x)=\frac{1}{x+\sqrt{1+x^2}}+\log x-\log( 1+\sqrt{1+x^2})\leq 1+\log x-\log 2,\qquad \forall x>0,
\]
which gives \eqref{eq:unif_log_est}.
Moreover, as $\log(1+ s) \le s$ for $s>-1$, we deduce that 
\begin{align*}
    \varphi(x)&  \le \frac{1}{x+\sqrt{1+x^2}}+\frac{x}{1+\sqrt{1+x^2}}-1 \\ 
    &=-\frac{x}{(x+\sqrt{1+x^2})(1+\sqrt{1+x^2})},
\end{align*}
whence \eqref{eq:decay_est} follows immediately.
\end{proof}

\begin{proposition}\label{prop:Z}
For all $\varepsilon \geq 0$ there exists $C_\varepsilon>0$ such that, for all $t \geq 1$,
\begin{align}
\label{LinfZeta}
  \sup_{n \in \ZZ} e^{\varepsilon |n|/\sqrt{t}} h^{\ZZ}_t(n) &\leq \frac{C_{\varepsilon}}{\sqrt{t}} , \\
\label{L1Zeta}
  \sum_{n \in \ZZ} e^{\varepsilon |n|/\sqrt{t}} h^{\ZZ}_t(n) &\leq C_{\varepsilon} .
\end{align}
\end{proposition}
\begin{proof}
By \cite[Proposition 2.3]{CMS}, for all $t>0$ and $n \in \ZZ$,
\begin{equation}\label{est: grig}
  h_t^{\ZZ}(n) \approx \begin{cases} \frac{e^{ |n|\varphi(t/|n|)}}{(|n|+t)^{1/2}} &\text{if } n \neq 0, \\ 
  (1+t)^{-1/2} &\text{if } n=0, 
  \end{cases} 
\end{equation}
where $\varphi$ is the function defined in \eqref{fi}.

Fix now $\varepsilon \geq 0$. By applying Lemma \ref{l: stimaphi} with $x_0= 1/e^{\varepsilon+1}$ we obtain that, for every $n \in \NN_+$ and $t>0$,
\[
\varphi(t/n) \leq \begin{cases}
-\varepsilon -\log 2 &\text{if } t/n \leq 1/e^{\varepsilon+1}, \\
-C_0 n/t &\text{if } t/n \geq 1/e^{\varepsilon+1}, 
\end{cases}
\]
where the constant $C_0>0$ may depend on $\varepsilon$; thus, by \eqref{est: grig},
\[
    h^{\ZZ}_t(n) \lesssim \frac{1}{(n+t)^{1/2}} \begin{cases} 
		e^{-\varepsilon n} 2^{-n} 
		&\text{if } n \ge e^{\varepsilon+1} t,  \\ 
    e^{-C_0 n^2/t} &\text{if } n \le e^{\varepsilon+1} t,
    \end{cases}
\]
and therefore, if we also assume $t \geq 1$,
\begin{equation}\label{stimanuova}
   e^{\varepsilon n/\sqrt{t}} h^{\ZZ}_t(n) \lesssim t^{-1/2} \begin{cases}
	 2^{-n}, &\text{if } n \ge e^{\varepsilon+1} t, \\ 
   e^{-C n^2/t}, &\text{if } n \le e^{\varepsilon+1} t
   \end{cases}
\end{equation}
for some other constant $C> 0$ depending on $\varepsilon$. From \eqref{stimanuova} and the case $n=0$ of \eqref{est: grig} we immediately deduce \eqref{LinfZeta} for all $t \geq 1$, 
and moreover
\[\begin{split}
    \sum_{n \in \ZZ} e^{\varepsilon |n|/\sqrt{t}} h^{\ZZ}_t(n) &\lesssim h^{\ZZ}_t(0)+    \sum_{k=1}^\infty \bigg( 2^{-k}+t^{-1/2} e^{-C' k^2/t}\bigg) \\ 
    &\lesssim 1+t^{-1/2} \int_1^{\infty}  e^{-C'x^2/t} \, dx \lesssim 1,
\end{split}\]
which gives \eqref{L1Zeta}.
\end{proof}

\begin{remark}
One may wonder whether the uniform bound \eqref{LinfZeta} for $h_t^\ZZ$ may be improved to a super-exponential decay estimate of the form
\[
h^{\ZZ}_t(n) \lesssim t^{-1/2} e^{-\varepsilon (|n|/\sqrt{t})^\alpha} \qquad\forall t\geq 1, n \in \ZZ,
\]
for some $\alpha > 1$ and $\varepsilon>0$, i.e., a gaussian-type bound for large time. However, this is not the case, so in these respects the bound \eqref{LinfZeta} is optimal. Indeed, from \eqref{fi} one easily deduces that there is $\kappa \in (0,1)$ sufficiently small that $-\varphi(x) \approx \log(1/x)$ for $x \in (0,\kappa]$; thus, from \eqref{est: grig}, it follows that there is $c>0$ such that
\[
h_t^\ZZ(n) \gtrsim n^{-1/2} e^{-c n\log(n/t)} \qquad \forall n \geq t/\kappa >0.
\]
This lower bound is incompatible with the above gaussian-type upper bound whenever $\alpha>1$, as one can see, e.g., by comparing them in the region where $t \simeq n^\delta$ for some sufficiently small $\delta>0$. We point out that a super-exponential decay can be recovered if the heat semigroup $\{e^{-t\Delta}\}_{t>0}$ is replaced by its discrete-time counterpart $\{(1-\Delta)^k\}_{k \in \NN_+}$; gaussian-type bounds for the latter are known to hold in greater generality than for $\ZZ$, and sometimes are also referred  to as ``heat kernel bounds'' in the literature on graphs, see, e.g., \cite{Grigoryan,HSa}.
\end{remark}

\begin{remark}
When $\varepsilon=0$, the estimate \eqref{L1Zeta} is of course a consequence of the $\ell^1$-contractivity stated in \eqref{eq: htHt}, which holds for all $t>0$. However, when $\varepsilon>0$, the estimate \eqref{L1Zeta} does not hold for small $t$, since $ h^{\ZZ}_t(n) \approx t^{|n|}$ as $t \to 0^+$ for every fixed $n \in \ZZ$ (see \cite[Proposition 2.3]{CMS}).
\end{remark}

\subsection{Weighted \texorpdfstring{$L^1$}{L1}-estimates of \texorpdfstring{$H_t$}{Ht}}
In the next lemma, we provide pointwise estimates of the heat kernel $H_t$ on $\TT_{q+1}$ and its gradient, in terms of the heat kernel $h_t^\ZZ$ on $\ZZ$.

\begin{lemma}\label{lem: gradient}
For every $t>0$ and $x,y\in \TT_{q+1}$,
\begin{equation}\label{s: stimaHt}
H_t(x,y) \approx Q(x,y) \,q^{-d(x,y)/2}  \frac{d(x,y)+1}{t}h^{\ZZ}_{t}(d(x,y)+1)
\end{equation}
and
\begin{equation}\label{s: stimagHt}
|\nabla_x H_t(x,y)|    \lesssim Q(x,y) \,q^{-d(x,y)/2} h^{\ZZ}_{t}(d(x,y)+1)
\begin{cases}
   \displaystyle \frac{1}{t} \left(\frac{d(x,y)^2}{t}+ 1\right)  &\text{if } y \not\le x,\\[1em]
   \displaystyle \frac{d(x,y)+1}{t} &\text{if } y \le x.
\end{cases}
\end{equation}
The implicit constants in the above estimates do not depend on $q$.
\end{lemma}
\begin{proof}
The estimate \eqref{s: stimaHt} follows from \eqref{eq: H=qJ} and \eqref{eq: approx1}.

Fix now $t>0$ and $x,y \in \TT_{q+1}$ such that $y \not \le x$ and set $j=d(x,y)$, so that $d(\prd(x),y)=j-1$.
In this case, $Q(x,y)q^{-d(x,y)/2}=Q(\prd(x),y)q^{-d(\prd(x),y)/2}$, thus, by \eqref{eq: H=qJ},
\begin{multline*}
   \nabla_x H_t(x,y) =  H_t(x,y) - H_t(\prd(x),y) \\
		=\frac{2Q(x,y)}{t}\sum_{k=0}^\infty q^{-(j+2k)/2} \left( (j+2k+1) h^{\ZZ}_t(j+2k+1)\!\!-\!\!(j+2k)h^{\ZZ}_t(j+2k) \right).
\end{multline*}
Since $h_t^{\ZZ}$ is decreasing on $\NN$, by using \ref{en:cmsii} in Proposition \ref{prop: cms}, we deduce, for every $n \in \NN$, that
\begin{equation*}
\begin{split}
      h^{\ZZ}_t(n +1) &\geq (n+1) h^{\ZZ}_t(n+1)-n h^{\ZZ}_t(n)\\
      &\geq (n+1) \left( h^{\ZZ}_t(n+2)-h^{\ZZ}_t(n)\right)
			=-\frac{2(n+1)^2}{t}h^{\ZZ}_t(n+1),
\end{split}
\end{equation*}
thus
\[
|(n+1) h^{\ZZ}_t(n+1)-n h^{\ZZ}_t(n)| \lesssim \left(1+ \frac{(1+n)^2}{t}\right) h_t^\ZZ(n+1).
\]
Hence, 
\[
\begin{split}
|\nabla_x H_t(x,y)| &\lesssim \frac{Q(x,y)}{t} \sum_{k=0}^\infty q^{-(j+2k)/2} \left(1+ \frac{(j+2k)^2}{t}\right) h_t^\ZZ(j+2k+1) \\
&\approx \frac{Q(x,y)}{t} q^{-j/2} \left(1+ \frac{j^2}{t}\right) h_t^\ZZ(j+1),
\end{split}
\]
which gives the first estimate in \eqref{s: stimagHt}.

Suppose now instead that $y \le x$. Then $d(\prd(x),y) = d(x,y)+1$, and moreover $Q(\prd(x),y) q^{-d(\prd(x),y)/2} \leq Q(x,y) q^{-d(x,y)/2}$. As
\[
    |\nabla_x H_t(x,y)| \le H_t(x,y)+H_t(\prd(x),y),
\]  
the second estimate in \eqref{s: stimagHt} simply follows from \eqref{s: stimaHt}.
\end{proof}

We now state a technical lemma which we will repeatedly apply to compute integrals on $\TT_{q+1}$.

\begin{lemma}\label{lem:spheres} For every $k \in \NN$ and $y \in \TT_{q+1}$,
\begin{equation}\label{eq:weighted_sphere}
\sum_{x \in S_k(y)} q^{[\ell(x)-\ell(y)]/2} \approx q^{k/2}(k+1),
\end{equation}
where $S_k(y) = \{ x \in \TT_q \tc d(x,y) = k\}$ is the sphere centred at $y$ of radius $k$. Moreover,
\begin{equation}\label{eq:weighted_sphere_rest}
\sum_{x \in S_k(y) \tc x \leq y \text{ or } y \leq x} q^{[\ell(x)-\ell(y)]/2} \approx q^{k/2}
\end{equation}
and
\begin{equation}\label{eq:weighted_sphere_rlev}
\sup_{m \in \ZZ} \sum_{x \in S_k(y) \tc \ell(x) = m} q^{[\ell(x)-\ell(y)]/2} \approx q^{k/2}.
\end{equation}
The implicit constants in the above estimates are independent of $q$.
\end{lemma}
\begin{proof}
We can decompose $S_k(y) = \bigcup_{j=0}^k S_k^{(j)}(y)$, where
\[
S^{(j)}_k(y) \defeq \begin{cases}
\{x \in S_k(y) \tc x \le y\} &\text{if } j=0, \\
\{x \in S_k(y) \tc x \le \prd^{j}(y), \, x \not \le \prd^{j-1}(y)\} &\text{if } j=1,\dots,k.
\end{cases}
\]
It is easily seen that $\# S^{(j)}_k(y) \approx q^{k-j}$ and that, for every $x \in S^{(j)}_k(y)$, we have $\ell(x)-\ell(y) = 2j-k$. It follows that 
\begin{equation}\label{eq:splitting}
\sum_{x \in S_k(y)} q^{[\ell(x)-\ell(y)]/2} = \sum_{j=0}^{k} q^{j-k/2} \# S^{(j)}_k(y) \approx q^{k/2}(k+1),
\end{equation}
which proves \eqref{eq:weighted_sphere}. As for \eqref{eq:weighted_sphere_rest} and \eqref{eq:weighted_sphere_rlev}, it is enough to observe that
\begin{align*}
\{ x \in S_k(y) \tc x \leq y \text{ or } y \leq x \} &= S_k^{(0)} \cup S_k^{(k)}, \\
\{ x \in S_k(y) \tc \ell(x) = m \} &= \begin{cases} S_k^{(j)} &\text{if } j = (m-\ell(y)+k)/2 \in \{0,\dots,k\}, \\ \emptyset &\text{otherwise}, \end{cases} 
\end{align*}
so one can argue much as in \eqref{eq:splitting}, but with the sum in $j$ restricted to at most two summands instead of $k+1$, which leads to the improved estimates.
\end{proof}

We can now prove the weighted $L^1$-estimates for the heat kernel on $\TT_{q+1}$ stated in the introduction.

\begin{proof}[Proof of Theorem \ref{L1estimateHtgradHt}]
Fix $\varepsilon >0$, $t\geq 1$ and $y\in \TT_{q+1}$. 

To prove \eqref{stimapesataHt} we apply \eqref{s: stimaHt} and \eqref{eq:weighted_sphere} to obtain that
\begin{equation}\label{eq:heat_estimate}
\begin{split}
 &\sum_{x \in \TT_{q+1}} |H_t(x,y)|  \,e^{\varepsilon d(x,y)/\sqrt{t}} \mu(x) \\
&\lesssim \sum_{k=0}^{\infty} q^{-k/2} \frac{k+1}{t}h^{\ZZ}_t(k+1) \, e^{\varepsilon k/\sqrt{t}}  \sum_{x \in S_k(y)} q^{[\ell(x)-\ell(y)]/2} \\
 &\lesssim \sum_{k=0}^{\infty} \frac{(k+1)^2}{t}  e^{\varepsilon k/\sqrt{t}}  h^{\ZZ}_t(k+1) \\
 &\lesssim 1,
\end{split}
\end{equation}
where Proposition  \ref{prop:Z} was applied in the last step (since $t \geq 1$ here, any power of $(1+k)/\sqrt{t}$ can be absorbed into the exponential factor $e^{\varepsilon k/\sqrt{t}}$ simply by taking a slightly larger $\varepsilon$).

To prove \eqref{stimapesatanablaxHt} we split the sum as follows:
\[
\sum_{x \in \TT_{q+1}} | \nabla_x H_t(x,y)| \, e^{\varepsilon d(x,y)/\sqrt{t}} \mu(x) = \sum_{x \tc y \leq x} +\sum_{x \tc y \not\leq x}.
\]
For the part where $y \not\leq x$ we can argue much as in \eqref{eq:heat_estimate}, but using the first estimate in \eqref{s: stimagHt} in place of \eqref{s: stimaHt}, thus gaining an extra factor $t^{-1/2}$. 
In the region where $y \leq x$, instead, the second estimate in \eqref{s: stimagHt} does not give any gain compared to \eqref{s: stimaHt}; however, here we can use the improved bound \eqref{eq:weighted_sphere_rest} in place of \eqref{eq:weighted_sphere} and obtain that
\begin{equation}\label{eq:heat_estimate_rest}
\begin{split}
&\sum_{x  \tc y \leq x} | \nabla_x H_t(x,y)| \,e^{\varepsilon d(x,y)/\sqrt{t}} \mu(x) \\
&\lesssim \sum_{k=0}^{\infty} q^{-k/2} \frac{k+1}{t}h^{\ZZ}_t(k+1) \, e^{\varepsilon k/\sqrt{t}}  \sum_{x \in S_k(y) \tc y \leq x} q^{[\ell(x)-\ell(y)]/2} \\
 &\lesssim \sum_{k=0}^{\infty} \frac{k+1}{t} \, e^{\varepsilon k/\sqrt{t}}  h^{\ZZ}_t(k+1) \\
 &\lesssim t^{-1/2},
\end{split}
\end{equation}
as desired.

Since $H_t(x,y) = H_t(y,x)$, 
from \eqref{s: stimagHt} we also deduce a pointwise estimate for $|\nabla_y H(x,y)|$, with the roles of $x$ and $y$ reversed; 
so we can prove \eqref{stimapesatanablayHt} in much the same way as \eqref{stimapesatanablaxHt}, apart from the fact that here the sum must be split according to whether $x \leq y$ or $x \not\leq y$.

It remains to prove \eqref{stimanablaynablaxHt}.
By the semigroup property of $e^{-t\opL}$,
\[
     H_t(x,y)=\sum_{v \in \TT_{q+1}} H_{t/2}(x,v)H_{t/2}(v,y) \mu(v), \qquad x,y \in \TT_{q+1}.
\]
Thus, 
\[
  \nabla_y\nabla_x H_t(x,y) 
	=\sum_{v \in \TT_{q+1}} \nabla_x H_{t/2}(x,v) \nabla_y H_{t/2}(v,y) \mu(v)
\]
and, by the triangle inequality,
\begin{equation}\label{eq:heat_estimate_young}
\begin{split}
  &\sum_{x \in \TT_{q+1}} |\nabla_y\nabla_x H_t(x,y)| \, e^{\varepsilon d(x,y)/\sqrt{t}} \mu(x) \\
	&\leq \sum_{v \in \TT_{q+1}} |\nabla_y H_{t/2}(v,y)| \, e^{\varepsilon d(v,y)/\sqrt{t}} \mu(v) 
	 \sum_{x \in \TT_{q+1}} |\nabla_x H_{t/2}(x,v)| \, e^{\varepsilon d(x,v)/\sqrt{t}} \mu(x) \\
	&\lesssim t^{-1},
\end{split}
\end{equation}
where \eqref{stimapesatanablaxHt} and \eqref{stimapesatanablayHt} were applied.
\end{proof}

As a by-product of the above estimates, we can also show that, when the sums in Theorem \ref{L1estimateHtgradHt} are restricted to a horocycle, one gains extra decay in $t$.

\begin{proposition}
For all $\varepsilon \geq 0$ there exists $C_\varepsilon>0$ such that, for all $t \geq 1$,
\begin{align}
&\sup_{m \in \ZZ} \sup_{y \in \TT_{q+1}} \sum_{x \in \TT_{q+1} \tc \ell(x) = m} |H_t(x,y)| \, e^{\varepsilon d(x,y)/\sqrt{t}}\mu(x) \le \frac{C_\varepsilon}{\sqrt{t}}, 
 \label{rstimapesataHt} \\
&\sup_{m \in \ZZ} \sup_{y \in \TT_{q+1}} \sum_{x \in \TT_{q+1} \tc \ell(x) = m} |\nabla_x H_t(x,y)| \, e^{\varepsilon d(x,y)/\sqrt{t}} \mu(x) \le \frac{C_\varepsilon}{t},
\label{rstimapesatanablaxHt} \\
&\sup_{m \in \ZZ} \sup_{y \in \TT_{q+1}} \sum_{x \in \TT_{q+1} \tc \ell(x) = m} |\nabla_y H_t(x,y)| \, e^{\varepsilon d(x,y)/\sqrt{t}} \mu(x) \le \frac{C_\varepsilon}{t},
\label{rstimapesatanablayHt} \\
&\sup_{m \in \ZZ} \sup_{y \in \TT_{q+1}} \sum_{x \in \TT_{q+1} \tc \ell(x) = m} |\nabla_y \nabla_x H_t(x,y)| \, e^{\varepsilon d(x,y)/\sqrt{t}} \mu(x) \le \frac{C_\varepsilon}{t\sqrt{t}}.
\label{rstimapesatanablaxnablayHt}
\end{align}
Moreover, the constant in the estimates above does not depend on $q$.
\end{proposition}
\begin{proof}
By following the proof of Theorem \ref{L1estimateHtgradHt}, but using the improved bound \eqref{eq:weighted_sphere_rlev} in place of \eqref{eq:weighted_sphere}, one proves \eqref{rstimapesataHt}, as well as the analogues of \eqref{rstimapesatanablaxHt} and \eqref{rstimapesatanablayHt} where the sums are restricted to $y \not\leq x$ and $x \not\leq y$ respectively.

For the remaining parts of \eqref{rstimapesatanablaxHt} and \eqref{rstimapesatanablayHt}, one simply observes that, for fixed $m \in \ZZ$ and $y \in \TT_{q+1}$, the set of the $x \in \TT_{q+1}$ such that $\ell(x) = m$ and either $y \leq x$ or $x \leq y$ is made of points with $d(x,y) = |m-\ell(y)|$, that is, the distance $d(x,y)$ takes at most one value under those constraints.
As a consequence, one obtains an estimate analogous to \eqref{eq:heat_estimate_rest}, where however the sums in $k$ are restricted to a single value of the summation index, instead of ranging over the whole $\NN$;
by using the uniform bound \eqref{LinfZeta} for $h_t^\ZZ$ instead of the $\ell^1$-bound \eqref{L1Zeta}, one then gains an extra factor $t^{-1/2}$, as desired.

Finally, to prove \eqref{rstimapesatanablaxnablayHt} one can proceed much as in \eqref{eq:heat_estimate_young}, using \eqref{rstimapesatanablaxHt} in place of \eqref{stimapesatanablaxHt} to control the inner sum in $x$.
\end{proof}

\section{Boundedness of the Riesz transform}\label{s: proof}

In this section we prove Theorem \ref{t: teoRiesz}. The proof will be based on the following result proved in \cite[Theorem 1.2]{hs} and \cite[Theorem 3]{ATV1}
(see also \cite[Theorem 5.8]{LSTV}).

\begin{proposition}\label{lemmahebisch}
Let $\cK$ be an integral operator bounded on $L^2(\mu)$ such that $\cK=\sum_{n\in\ZZ}\cK_n$, where the series converges in the strong topology of $L^2(\mu)$.
Assume that $K_n$ is the integral kernel of $\cK_n$ and that there exist constants $C>0, c \in (0,1)$ and $a,b>0$ such that
\begin{align}
    &\sum_{x \in \TT_{q+1}} |K_n(x,y)| \, (1+c^{n}d(x,y))^a \, \mu(x) \le C \qquad \forall y \in \TT_{q+1} \label{Kn},\\ 
     &\sum_{x \in \TT_{q+1}} |K_n(x,y)-K_n(x,z)| \, \mu(x) \le C(c^{n}d(y,z))^b \qquad \forall y,z \in \TT_{q+1}. \label{Knseconda}
\end{align}
Then, $\cK$  is of weak type $(1,1)$, bounded on $L^p(\mu)$ for every $p\in (1, 2]$, and bounded from $H^1(\mu)$ to $L^1(\mu)$.
\end{proposition}
  
\begin{remark}\label{rem:cz_gradient}
When $b=1$, the condition \eqref{Knseconda} is equivalent to the following estimate:
\begin{equation}
     \sup_{y \in \TT_{q+1}} \sum_{x \in \TT_{q+1}} |\nabla_y K_n(x,y) | \, \mu(x) \le C\, c^{n} . \label{Knsecondagrad}
\end{equation}
This fact can be easily seen by applying the following elementary lemma. 
\end{remark}

\begin{lemma}\label{l: nabla_yF}
Let $F : \TT_{q+1} \times \TT_{q+1} \to \CC$ and set
\[
C_0 = \sup_{y\in\TT_{q+1}} \sum_{x\in\TT_{q+1}} |\nabla_y F(x,y)| \, \mu(x).
\]
Then
\[
\sum_{x \in \TT_{q+1}} |F(x,y)-F(x,z)| \, \mu(x)\leq C_0 \, d(y,z) \qquad\forall y,z \in \TT_{q+1}.
\]
\end{lemma}
\begin{proof} 
Consider a path $x_0,\dots,x_d$ joining $y$ to $z$, with $d=d(y,z)$.
As $x_i \sim x_{i+1}$, we either have $x_i \in \scc(x_{i+1})$ or $x_i=\prd(x_{i+1})$ for every $i=0,\dots,d-1$. Thus, 
\[
 \sum_{x\in \TT_{q+1}} |F(x,y)-F(x,z)| \, \mu(x)\leq   \sum_{i=0}^{d-1}\sum_x|F(x,x_i)-F(x,x_{i+1})| \, \mu(x)\leq  C_0\, d(y,z),
\]
as required.
\end{proof} 

We aim to apply Proposition \ref{lemmahebisch} to the Riesz transform $\Rz$, whose integral kernel with respect to the measure $\mu$ is 
\[
R(x,y)= \frac{1}{\sqrt{\pi}}\int_0^\infty t^{-1/2} \nabla_x H_t(x,y) \, dt ,  \qquad x,y \in \TT_{q+1}.
\]
The part of the integral corresponding to small $t$ is easily dealt with.

\begin{lemma}\label{lem:localriesz}
The operator $\Rz^{(0)}$ with integral kernel
\[
    R^{(0)}(x,y)=\frac{1}{\sqrt{\pi}}\int_0^1 t^{-1/2} \nabla_x H_t(x,y) \, dt,
\]
is bounded on $L^p(\mu)$ for all $p\in [1,\infty]$.
\end{lemma}
\begin{proof}
By \eqref{contractive},
\[
\begin{split}
\sup_{y\in\TT_{q+1}}\sum_{x \in \TT_{q+1}} |R^{(0)}(x,y)| \, \mu(x)
&\lesssim \int_0^1 t^{-1/2}\sum_{x \in \TT_{q+1}} |\nabla_x H_t(x,y)| \, \mu(x) \,dt\\
&\lesssim  \int_0^1 t^{-1/2}\sum_{x \in \TT_{q+1}} |H_t(x,y)| \, \mu(x) \,dt\\
&=  \int_0^1 t^{-1/2} \,dt  \approx 1,
\end{split}
\]
and a similar estimate holds with the roles of $x$ and $y$ reversed.
\end{proof}

Now, for every $n \in \NN$, we define 
\[
    K_n(x,y)=\frac{1}{\sqrt{\pi}}\int_{2^n}^{2^{n+1}}t^{-1/2} \nabla_x H_t(x,y) \, dt,
\]
so that $R(x,y)=R^{(0)}(x,y)+\sum_{n \in \NN} K_n(x,y)$. Notice that
\[
\nabla_y K_n(x,y)=\frac{1}{\sqrt{\pi}}\int_{2^n}^{2^{n+1}}t^{-1/2} \nabla_y \nabla_x H_t(x,y) \, dt.
\]
So, by integrating in $t$ the estimates of Theorem \ref{L1estimateHtgradHt}, we immediately deduce the following statement.

\begin{lemma}\label{lemmahs}
The following estimates hold for every $n \in \NN$ and $\varepsilon \geq 0$: 
\begin{gather}
\label{1}
   \sup_{y \in \TT_q} \sum_{x \in \TT_{q+1}} |K_n(x,y)| \, e^{\varepsilon d(x,y)/2^{n/2}} \mu(x) \lesssim 1, \\
\label{2}
       \sup_{y \in \TT_q} \sum_{x \in \TT_{q+1}} |\nabla_y K_n(x,y)| \, e^{\varepsilon d(x,y)/2^{n/2}} \mu(x) \lesssim 2^{-n/2}.
\end{gather}
The implicit constants in the above estimates may depend on $\varepsilon$, but not on $q$.
\end{lemma}

We can finally complete the proof of the Riesz transform boundedness result.

\begin{proof}[Proof of Theorem \ref{t: teoRiesz}]
As $\opL = \frac{1}{2} \nabla^* \nabla$, clearly $\| \opL^{1/2} f \|_{L^2(\mu)} = 2^{-1/2} \| \nabla f \|_2$, thus $\Rz = \nabla \opL^{-1/2}$ is trivially bounded on $L^2(\mu)$. So, by Lemma \ref{lem:localriesz}, also $\Rz - \Rz^{(0)}$ is. By combining Lemma \ref{lemmahs} with Proposition \ref{lemmahebisch} and Remark \ref{rem:cz_gradient}, we obtain that $\Rz-\Rz^{(0)}$ is of weak type $(1,1)$, bounded on $L^p(\mu)$ for all $p \in (1,2]$, and bounded from $H^1(\mu)$ to $L^1(\mu)$. Since the same properties are true for $\Rz^{(0)}$ by Lemma \ref{lem:localriesz}, the desired boundedness of $\Rz$ follows.
\end{proof}

\section*{Final remarks}
  
In \cite{LSTV} the Calder\'on--Zygmund theory of \cite{hs}, as well as the Hardy and BMO spaces of \cite{ATV2, ATV1}, were generalized to a nonhomogeneous tree $T$ of bounded degree equipped with an arbitrary locally doubling flow $m$. The metric measure space $(T,d,m)$ is in general nondoubling and of exponential growth, although a global  Poincar\'e inequality holds, see \cite{LSTV2}. One can easily define the notion of flow gradient in this more general setting and introduce a natural flow Laplacian. Hence one could study the boundedness properties of the first order Riesz transform on $(T,d,m)$, following the strategy discussed in the section above. The main difficulty to face is the lack of explicit formulas for the heat kernel of a flow Laplacian and its gradient. Nevertheless it might be possible to transfer the weighted $L^1$-estimates for the heat kernel and its gradient from an appropriate homogeneous tree to the tree $T$ and consequently study the boundedness of the Riesz transform on $(T,d,m)$. This is actually a work in progress \cite{MSTV}.

\section*{Acknowledgments}
This work was partially supported by the Progetto ``Harmonic analysis on continuous and discrete structures''(bando Trapezio Compagnia San Paolo).
The first-named author gratefully acknowledges the support of Compagnia di San Paolo through a Starting Grant at Politecnico di Torino.
 The second- and third-named authors were also partially supported by the INdAM--GNAMPA 2022 Project ``Generalized Laplacians on continuous and discrete structures'' (CUP\_E55F22000270001).
The authors are members of the Gruppo Nazionale per l'Analisi Matema\-tica, la Probabilit\`a e le loro Applicazioni (GNAMPA) of the Istituto Nazionale di Alta Matematica (INdAM).

\end{document}